\documentclass[a4paper,twoside]{article}
\usepackage{amssymb}
\usepackage{amsmath}
\usepackage{mathrsfs}
\usepackage{amsfonts}
\usepackage{color}
\usepackage{vmargin}
\usepackage{amsthm}
\usepackage{ esint }
\usepackage{a4}
\usepackage{graphicx}

\setmarginsrb{30mm}{20mm}{30mm}{20mm}{10mm}{10mm}{10mm}{10mm}

\theoremstyle{plain}
\newtheorem{theorem}{Theorem}[section]
\newtheorem{thm}[theorem]{Theorem}
\newtheorem{lem}{Lemma}[section]
\newtheorem{cor}{Corollary}[section]
\newtheorem{prop}{Proposition}[section]

\theoremstyle{definition}
\newtheorem{defn}{Definition}[section]
\newtheorem{ex}{Example}
\newtheorem{rem}{\textnormal{\textbf{Remark}}}

\newcommand{\norm}[1]{\|{#1}\|}

\newcommand{\ud}{\mathrm{d}}

\newcommand{\bbR}{{\mathbb{R}}}
\newcommand{\bbN}{{\mathbb{N}}}
\def \R {{\mathbb {R}}}

\newcommand{\bbG}{\mathbb{G}}

\newcommand{\ep}{\varepsilon}

\newcommand{\average}{{\mathchoice {\kern1ex\vcenter{\hrule height.4pt
width 6pt
depth0pt} \kern-9.7pt} {\kern1ex\vcenter{\hrule height.4pt width 4.3pt
depth0pt}
\kern-7pt} {} {} }}

\newcommand{\calB} {\mathcal{B}}

\newcommand{\Om}{\Omega}
\newcommand{\Hei}{{\mathbb{H}_1}}
\newcommand{\G}{{\mathbb{G}}}

\def\vint_#1{\mathchoice%
          {\mathop{\kern 0.2em\vrule width 0.6em height 0.69678ex depth -0.58065ex
                  \kern -0.8em \intop}\nolimits_{\kern -0.4em#1}}%
          {\mathop{\kern 0.1em\vrule width 0.5em height 0.69678ex depth -0.60387ex
                  \kern -0.6em \intop}\nolimits_{#1}}%
          {\mathop{\kern 0.1em\vrule width 0.5em height 0.69678ex
              depth -0.60387ex
                  \kern -0.6em \intop}\nolimits_{#1}}%
          {\mathop{\kern 0.1em\vrule width 0.5em height 0.69678ex depth -0.60387ex
                  \kern -0.6em \intop}\nolimits_{#1}}}
\def\vintslides_#1{\mathchoice%
          {\mathop{\kern 0.1em\vrule width 0.5em height 0.697ex depth -0.581ex
                  \kern -0.6em \intop}\nolimits_{\kern -0.4em#1}}%
          {\mathop{\kern 0.1em\vrule width 0.3em height 0.697ex depth -0.604ex
                  \kern -0.4em \intop}\nolimits_{#1}}%
          {\mathop{\kern 0.1em\vrule width 0.3em height 0.697ex depth -0.604ex
                  \kern -0.4em \intop}\nolimits_{#1}}%
          {\mathop{\kern 0.1em\vrule width 0.3em height 0.697ex depth -0.604ex
                  \kern -0.4em \intop}\nolimits_{#1}}}

\newcommand{\kintint}[2]{\mathchoice%
          {\mathop{\kern 0.2em\vrule width 0.6em height 0.69678ex depth -0.58065ex
                  \kern -0.8em \intop}\nolimits_{\kern -0.45em#1}^{#2}}%
          {\mathop{\kern 0.1em\vrule width 0.5em height 0.69678ex depth -0.60387ex
                  \kern -0.6em \intop}\nolimits_{#1}^{#2}}%
          {\mathop{\kern 0.1em\vrule width 0.5em height 0.69678ex depth -0.60387ex
                  \kern -0.6em \intop}\nolimits_{#1}^{#2}}%
          {\mathop{\kern 0.1em\vrule width 0.5em height 0.69678ex depth -0.60387ex
                  \kern -0.6em \intop}\nolimits_{#1}^{#2}}}

\renewcommand{\div}{\operatorname{div}}

\newcommand{\kom}[1]{}
\renewcommand{\kom}[1]{{\bf [#1]}}

\definecolor{blau}{rgb}{0.1,0.0,0.9}

\definecolor{lila}{rgb}{0.5,0.0,0.5}

\definecolor{Bcolor}{rgb}{0.5,0.0,0.0}

\newcounter{komcounter}
\numberwithin{komcounter}{section}

\makeatletter
\def\cleardoublepage{\clearpage\if@twoside \ifodd\c@page\else
\hbox{}
\thispagestyle{empty}
\newpage
\if@twocolumn\hbox{}\newpage\fi\fi\fi}
\makeatother

\begin{document}

\title {Variational approach to the asymptotic mean-value property for the $p$-Laplacian on Carnot groups}

\author{
Tomasz Adamowicz, Antoni Kijowski 
\\
\it\small Institute of Mathematics, Polish Academy of Sciences \\
\it\small ul. \'Sniadeckich 8, 00-656 Warsaw, Poland\/{\rm ;}
\it\small T.Adamowicz@impan.pl
\\
\\
Andrea Pinamonti
\\
\it\small Dipartimento di Matematica, Universit\`a di Trento \\
\it\small Via Sommarive 14, 38123 Povo,Trento, Italy\/{\rm ;}
\it\small andrea.pinamonti@unitn.it
\\
\\
Ben Warhurst 
\\
\it\small Institute of Mathematics,
\it\small University of Warsaw,\\
\it\small ul.Banacha 2, 02-097 Warsaw, Poland\/{\rm ;}
\it\small B.Warhurst@mimuw.edu.pl
}
\date{}
\maketitle


\begin{abstract}
 Let $1<p\leq \infty$. We provide an asymptotic characterization of continuous viscosity solutions $u$ of the normalized $p$-Laplacian $\Delta_{p\,\bbG}^Nu=0$ in any Carnot group $\bbG$.
\end{abstract}

%
%
%

\maketitle
\section{Introduction}

The study of mean-value properties of solutions of elliptic PDEs has a long and fruitful history. For harmonic functions in the Euclidean setting, the study goes back to Gauss, Koebe, Volterra, and Zaremba, to mention just a few, see also~\cite{adwa} for  recent results in Carnot groups. A generalized mean-value property originating in~\cite{MPR} and~\cite{MPR2}, called the asymptotic mean-value property, facilitates similar analysis of $p$-harmonic functions, one of the most important nonlinear counterparts of harmonic functions. Related are applications of $p$-harmonic functions in statistical Tug-of-War games, see for instance~\cite{MPR} and~\cite{PS}. In the setting of Carnot groups, similar studies have been conducted in~\cite{FLM1} and~\cite{FP}. 

 A new approach to the asymptotic mean-value property has been recently proposed in~\cite{Mag} (see also~\cite{BM} for relations with statistical games). More precisely, in~\cite{Mag}, the authors proved that every viscosity solution $u$ to the normalized $p$-laplacian in an open set $\Om\subset \R^n$ for a given $1\leq p\leq \infty$ (Definition \ref{def25}), can be characterized using an  asymptotic mean-value property in terms of the function $\mu_p(\varepsilon,u)(x)$, defined as the unique minimizer of the following variational problem
\begin{equation*}
\left\|u-\mu_p(\varepsilon,u)\right\|_{L^p(\overline{B_{\varepsilon}(x)})}=\min_{\lambda\in\bbR}\|u-\lambda\|_{L^p(\overline{B_{\varepsilon}(x)})},
\end{equation*}
where $B_{\varepsilon}(x)\subset \Omega$ denotes the ball centered at $x$ with radius $\varepsilon$. This notion encompasses the median, the mean-value and the min-max mean of a continuous function, see~\cite{Mag} for details.

In the present paper we generalize the results of \cite{Mag} to the setting of an arbitrary Carnot group.  

Let $\bbG$ be a Carnot group of step $k$ (Definition~\ref{defiC}). Denote by $\Delta_{p, \G}^N$ the subelliptic normalized $p$-Laplacian (see \eqref{p-norm-lapl} and \eqref{infty-norm-lapl}) and by $\mu_p(\varepsilon, u)$ the generalized median of a function $u$ defined uniquely as in~\eqref{minimizer-mu-p}. The theorem below stays that a viscosity solution of $\Delta_{p, \G}^Nu=0$ can be characterized asymptotically by the minimum $\mu_p(\varepsilon, u)$. This provides one more, intrinsic, way to characterize $p$-harmonic functions via a variant of the asymptotic mean-value property.

\begin{thm}\label{main-thm}
Let $1< p\leq \infty$ and let $\Omega\subset\bbG$ be open. For a function $u\in C^0(\Omega)$ the following are equivalent:
\begin{itemize}
	\item[(i)] $u$ is a viscosity solution of $\Delta_{p,\G}^N u=0$ in $\Omega$;
	\item[(ii)] $u(x)=\mu_p(\varepsilon,u)(x)+o(\varepsilon^2)$ as $\varepsilon\to 0$, in the viscosity sense for every $x\in\Omega$.
\end{itemize}
\end{thm}

In order to prove this theorem we first prove Lemma~\ref{lem-main-cc}, where the asymptotic behavior of minimizers $\mu_p$ is described for quadratic polynomials on balls. We illustrate the discussion with examples of the Heisenberg  group and Carnot groups of step 2, see Examples 3 and 4 in Section 3. As presented in Remark~\ref{rem-2.15} in Section 3, our results generalize those obtained in the Euclidean setting in~\cite{Mag}. The techniques employed in \cite{Mag} do not allow us to include in our discussion the case $p=1$, see Remark~\ref{ex-fail} at the end of Section 3.

\section{Carnot groups}\label{Carnot2}
In what follows, we briefly recall some standard facts on Carnot groups, see \cite{libro,Cap,Gro, Mont} for a more detailed treatment.
\begin{defn}\label{defiC}
A finite dimensional Lie algebra $\mathfrak{g}$, is said to be stratified of step $k\in\bbN$, if there exists subspaces $V_1,\ldots,V_k$ of $\mathfrak{g}$   such that:
\begin{align*}
	\mathfrak{g}=V_1\oplus\cdots\oplus V_k\hbox{ and } [V_1,V_i]=V_{i+1}\ \ i=1,\ldots, k-1 ; \quad 
	[V_1,V_k]=\{0\}.
\end{align*} We denote by $v_k$ the dimension of $ V_k$.
 
A connected and simply connected Lie group $\bbG$ is a Carnot group if its Lie algebra $\mathfrak{g}$ is finite dimensional and stratified. We also set $h_0:=0$, $h_i:=\sum_{j=1}^i v_j$ and $m:=h_k$.
\end{defn}

Using the exponential map, every Carnot group $\bbG$ of step $k$ is isomorphic as a Lie group to $(\bbR^{m},\cdot)$ where $\cdot$ is the group operation given by the Baker-Campbell-Hausdorff formula. 

For each $x\in \bbG$ we define left translation by  $\tau_x:\bbG\longrightarrow\bbG$ by 
\begin{align*}
\tau_{x}(y):=x\cdot y.
\end{align*}
For each $\lambda>0$ we define a dilation $\delta_{\lambda}:\bbG\longrightarrow\bbG$ by
\begin{align*}
\delta_{\lambda}(x)=\delta_{\lambda}(x_1,\ldots, x_m):=(\lambda^{\sigma_1} x_1,\ldots, \lambda^{\sigma_k} x_m),
\end{align*} 
where $\sigma_i\in\bbN$ is called the homogeneity of the variable $x_i$ in $\bbG$ and it is defined by
$\sigma_j:=i$, whenever $h_{i-1}< j \leq h_i$.

We endow $\bbG$ with a pseudonorm and pseudodistance by defining
\begin{align}
	&|x|_{\bbG}:=|(x^{(1)},\ldots, x^{(k)})|_{\bbG}:=\Big(\sum_{j=1}^k \norm{x^{(j)}}^{\frac{2
	k!}{j}}\Big)^{\frac{1}{2k!}} \label{distanza} \\
&\ud(x,y):=|y^{-1}\cdot x|_{\bbG}, \nonumber
\end{align}
where $x^{(j)}:=(x_{h_{j-1}+1},\ldots, x_{h_j})$ and $\norm{x^{(j)}}$ denotes the standard Euclidean norm in $\bbR^{h_j-h_{j-1}}$. We define the pseudoball centered at $x\in\bbG$ of radius $R>0$ by \[B(x,R)=B_R(x):=\{y\in\bbG:\, \ |y^{-1}\cdot x|_{\bbG}<R\}.\] 

We illustrate the concept of Carnot groups with the following important examples.

\begin{ex}[The Heisenberg groups $\mathbb{H}_n$]\label{ex-heisenberg}

The $n$-dimensional Heisenberg group  $\G=\mathbb{H}_n$, is the Carnot group with a $2$-step Lie algebra and orthonormal basis $\{ X_1,\dots,X_{2n},Z\}$ such that
$$
\mathfrak{g}_1={\rm Span}\, \{X_1, \dots, X_{2n}\},  \quad \mathfrak{g}_2={\rm Span}\, \{Z\},
$$
and the nontrivial brackets are $[X_i,X_{n+i}]=Z$ for $i=1,\dots,n$.

In particular, if $n=1$, then the Heisenberg group $\Hei$ is often presented using coordinates $(z,t)$, where $z=x+iy \in \mathbb{C}$ and $t \in \mathbb{R}$, and multiplication defined by $(z_1,t_1)(z_2,t_2)=(z_1+z_2, t_1+t_2 + 2\, {\rm Im}\,(z_1 \bar z_2))$. The pseudonorm given by $\|(z,t)\| =  ( |z|^4 + t^2)^{1/4}$ gives rise to a left invariant distance defined by $d_{\Hei}(p,q)= \|p^{-1}q\|$ which is called the Heisenberg distance. A dilation by $r >0$ is defined by $\delta_r (z,t)  = ( rz, r^2 t)$ and the left invariant Haar measure $\lambda$ is simply the $3$-dimensional Lebesgue measure, moreover $\delta_r^*d \lambda=r^4 d \lambda$. It follows that the Hausdorff dimension of the metric measure space $(\Hei, d_{\Hei}, \lambda)$ is $4$, and the space is $4$-Ahlfors regular, i.e., there exists a positive constant $c$ such that for all balls $B$ with radius $r$, we have $\frac{1}{c}r^4  \leq \mathcal{H}^4(B) \leq c r^4$, where $\mathcal{H}^4$ denotes the $4$-dimensional Hausdorff measure induced by $d_{\Hei}$.
\end{ex}

The following proposition, proved in \cite{libro}, shows that the Lebesgue measure is the Haar measure on Carnot groups.
\begin{prop}\label{prop-leb-meas}
Let $\bbG=(\bbR^m,\cdot)$ be a Carnot group. Then the Lebesgue measure on $\bbR^m$ is invariant with respect to the left and the right translations on $\bbG$. Precisely, if we denote by $|E|$ the Lebesgue measure of a measurable set $E\subset \bbR^m$, then for all $x\in \bbG$ we have that $|x\cdot E|=|E|=|E\cdot x|$. Moreover, for all $\lambda>0$ it holds $\delta_{\lambda}(E)|=\lambda^Q|E|$, where $Q:=\sum_{j=1}^m v_j\sigma_j$.
\end{prop}
A basis  $X=\{X_1,\ldots, X_m\}$  of $\mathfrak{g}$, is called \emph{the Jacobian basis} if $X_j=J(e_j)$ where $(e_1,\ldots, e_m)$ is the canonical basis of $\bbR^m$ and $J:\bbR^m\longrightarrow \mathfrak{g}$ is defined by $J(\eta)(x):=\mathcal{J}_{\tau_x}(0)\cdot \eta$, where $\mathcal{J}_{\tau_x}$ denotes the Jacobian matrix of $\tau_x$.

Let us recall the following classical proposition describing the Jacobian basis on Carnot groups, see \cite[Corollary 1.3.19]{libro} for a proof.
\begin{prop}\label{lcampi}
Let $\bbG=(\bbR^m,\cdot)$ be a Carnot group of step $k\in\bbN$. Then the elements of the Jacobian basis $\{X_1,\ldots, X_m\}$ have polynomial coefficients and if $h_{l-1}<j\leq h_l$, $ 1\leq l\leq k$, then
\[
 X_j(x)=\partial_j+\sum_{i>h_l}^m a_i^{(j)}(x)\partial_i,
\]
where $a_i^{(j)}(x)=a_i^{(j)}(x_1,\ldots, x_{h_{l-1}})$ when $h_{l-1}<i\leq h_l$, and $a_i^{(j)}(\delta_{\lambda}(x))=\lambda^{\sigma_i-\sigma_j} a_i^{(j)}(x)$.
\end{prop}
The following definition is one of the key concepts of the analysis on Carnot groups. Let $X=\{X_1,\ldots, X_m\}$ be a Jacobian basis of $\bbG=(\bbR^m,\cdot)$. For any function $u\in C^1(\bbR^m)$, we define its \emph{horizontal gradient}  by the formula
\[
 \nabla_{V_1}u:=\sum_{i=1}^{h_1} (X_i u)X_i \quad 
\]
and the \emph{intrinsic divergence} of $u$ as
\[
\div_{V_1} u:=\sum_{i=1}^{h_1} X_i u.
\]
Moreover, for $2\leq j\leq k$, we set  $\nabla_{V_j}u:=\sum_{h_{j-1}<i\leq h_j} (X_i u)X_i$.
The horizontal Laplacian $\Delta_{\bbG}u$ of a function $u:\bbG\longrightarrow\bbR$ is defined by the following 
\begin{equation*}
		\Delta_{\bbG}u:= \sum_{i=1}^{h_1} X_i^2 u.
\end{equation*}
A priori, one studies solutions to the Laplace equation under the $C^2$-regularity assumption. However, as in the Euclidean setting, it is natural to weaken the required degree of regularity and consider weak solutions belonging to the so-called horizontal Sobolev space. For further details we refer to e.g.~\cite{cdg, man-ming}. 
The following results describe the Taylor expansion formula in the Carnot groups, see~\cite[Proposition 20.3.11]{libro} .
\begin{prop} Let $\Om\subset \bbG$ be an open neighborhood of $0$ and let $u\in C^{\infty}(\Omega)$. Then, the following Taylor formula holds for any point $P=(x^{(1)},x^{(2)},\ldots, x^{(k)}) \in \Om$: 
\[
u(P)=u(0)+\langle\nabla_{V_1} u(0), x^{(1)}\rangle_{\R^{h_1}}+\langle\nabla_{V_2} u(0), x^{(2)}\rangle_{\R^{h_2}}+\frac12\langle D_{V_1}^{2,*} u(0) x^{(1)}, x^{(1)}\rangle_{\R^{h_1}}+o(\|P\|^2)
\]
where 
\begin{align*}
	D_{V_1}^{2,*}u:=\left(\frac{(X_iX_j+X_jX_i)u}{2}\right)_{1\leq i,j\leq h_1}
\end{align*}
is the so called \emph{symmetrized horizontal Hessian} of $u$.
\end{prop}

Next, we recall the definition of the main differential operator studied in this work.  For $p\in [1,+\infty]$ the \emph{subelliptic normalized $p$-Laplace operator} is  
\begin{align}\label{p-norm-lapl}
	\Delta_{p, \bbG}^N u:=\frac{\div_{V_1}(|\nabla_{V_1} u|^{p-2}\nabla_{V_1} u)}{|\nabla_{V_1} u|^{p-2}}\quad \mbox{if}\quad 1\leq p<\infty
\end{align}
and 
\begin{equation}\label{infty-norm-lapl}
\begin{split}
\Delta_{\infty, \bbG}^Nu:= \frac{\Big\langle D_{V_1}^{2,*}u \,\,\frac{\nabla_{V_1}u}{|\nabla_{V_1}u|},\frac{\nabla_{V_1}u}{|\nabla_{V_1}u|}\Big\rangle}{|\nabla_{V_1} u|^2}.
\end{split}
\end{equation}
Note that for $p=2$, $\Delta_{2,\bbG} u=\Delta_{\bbG} u$ is the so called Kohn-Laplace operator in $\bbG$. Thus, the $p$-Laplace operator is the natural generalization of the Laplacian. Furthermore, the $\infty$-Laplacian can be viewed as a limit of $p$-Laplacians in the appropriate sense for $p\to \infty$.  Among its applications, let us mention best Lipschitz extensions, image processing and mass transport problems, see e.g. the presentation in~\cite{MPR} and references therein.

In the case of the non-renormalized $p$-Laplacian, notions of a viscosity solution and a weak solution agree for $1<p<\infty$, see~\cite{JLM} for the result in the Euclidean setting and~\cite{bie} for the Heisenberg group. Since the normalized $p$-Laplacian is in the non-divergence form, the concept of viscosity solutions is more handy than weak solutions.  Let us now introduce this notion.  
\begin{defn}\label{def25}
Fix a value of $p\in [1,\infty]$ and consider the subelliptic normalized $p$-Laplace equation 
\begin{equation}\label{plaplace}
\Delta_{p, \bbG}^N u=0 \qquad \mbox{in}\qquad \Omega\subset\bbG.
\end{equation}
\begin{itemize}
\item[(i)] A lower semi-continuous function $u$, is a viscosity supersolution of  (\ref{plaplace}), if for every $x_0\in\Omega$, and every $\phi\in C^2(\Omega)$ such that $\nabla_{V_1}\phi(x_{0})\not=0$ and $u-\phi$ has a strict minimum at $x_{0}\in\Omega$, we have
$
\Delta_{p,\bbG}^N \phi \leq 0
$
in $\Omega$.
\item[(ii)] A lower semi-continuous function $u$, is a viscosity subsolution of  (\ref{plaplace}), if for every $x_0\in\Omega$, and every $\phi\in C^2(\Omega)$ such that $\nabla_{V_1}\phi(x_{0})\not=0$ and $u-\phi$ has a strict minimum at $x_{0}\in\Omega$, we have
$
\Delta_{p,\bbG}^N \phi \geq 0
$
in $\Omega$.
\item[(iii)] A continuous function $u$ is a viscosity solution of of  (\ref{plaplace}),  if it is both a viscosity supersolution and a viscosity subsolution in $\Omega$.
\end{itemize}
\end{defn}
Fix an open set $\Omega\subset\bbG$, let $1\leq p\leq \infty$ and let $u$ be a real-valued continuous function in $\Omega$.  For a given $x\in \Om$, choose $\ep>0$ so that $\overline{B_{\ep}(x)}\subset \Om$,  we define the number $\mu_p(\varepsilon, u)(x)$ (or simply $\mu_p(\varepsilon, u)$ if the point $x$ is clear from the context)  as the unique real number satisfying
\begin{equation}\label{minimizer-mu-p}
\|u-\mu_p(\varepsilon,u)\|_{L^p(\overline{B_{\varepsilon}(x)})}=\min_{\lambda\in\bbR}\|u-\lambda\|_{L^p(\overline{B_{\varepsilon}(x)})}.
\end{equation}
The following properties of $\mu_p(\varepsilon, u)(x)$ have been proved in \cite{Mag} for the setting of compact topological spaces $X$, equipped with a positive Radon measure $\nu$ such that $\nu(X)<\infty$. Here we apply results from \cite{Mag} to $X=\overline{B_{\ep}(x)}\subset \G$ and $\nu$ the Lebesgue measure, cf. Proposition~\ref{prop-leb-meas}.

 In Theorem~\ref{prop} below, we summarize results proven in Theorems 2.1, 2.4 and 2.5 in~\cite{Mag}.
 
\begin{thm} \label{prop} Let $1\leq p\leq \infty$ and $u\in C(\overline{B_{\varepsilon}(x)})$. 
\begin{itemize}
	\item[(1)] There exists a unique real valued $\mu_p(\varepsilon,u)$ such that 
\[
\|u-\mu_p(\varepsilon,u)\|_{L^p(\overline{B_{\varepsilon}(x)})}=\min_{\lambda\in\bbR}\|u-\lambda\|_{L^p(\overline{B_{\varepsilon}(x)})}.
\]
Furthermore, for $1\leq p<\infty$, $\mu_p(\varepsilon,u)$ is characterized by the equation
\begin{equation}\label{char}
\int_{B_{\varepsilon}(x)}\left|u(y)-\mu_p(\varepsilon,u)\right|^{p-2}\left(u(y)-\mu_p(\varepsilon,u)\right)\, dy=0, 
\end{equation}
where for $1\leq p<2$ we assume that the integrand is zero if $u(y)-\mu_p(\varepsilon,u)=0$. For $p=\infty$ we have the following equality:
\begin{equation}\label{inf}
\mu_\infty (\varepsilon,u) = \frac12 \left( \min_{\overline{B(x,\varepsilon)}} u +\max_{\overline{B(x,\varepsilon)}} u \right).
\end{equation}
 \item[(2)] If $1\leq p \leq \infty$ then it follows that
\begin{equation*}
\left|\|u-\mu_p(\varepsilon,u)\|_{L^p(\overline{B_{\varepsilon}(x)})}-\|v-\mu_p(\varepsilon,v)\|_{L^p(\overline{B_{\varepsilon}(x)})}\right|\leq \|u-v\|_{L^p(\overline{B_{\varepsilon}(x)})}
\end{equation*}
for any $u,v\in L^p(\overline{B_{\varepsilon}(x)})$. Moreover, if $u_n\to u$ in $L^p(\overline{B_{\varepsilon}(x)})$ for $1\leq p\leq \infty$ and $u_n,u\in C^0(\overline{B_{\varepsilon}(x)})$ for $p=1$, then $\mu_p(\varepsilon,u_n)\to \mu_p(\varepsilon,u)$ as $n\to\infty$, the same is true for any $p\in [1,\infty]$ if $\{u_n\}\subset C^0(\overline{B_{\varepsilon}(x)})$ converges uniformly on $\overline{B_{\varepsilon}(x)}$ as $n\to \infty$.
\item[(3)] Let $u$ and $v$ be two functions which, in the case $1<p\leq \infty$, belong to   $L^p(B_{\varepsilon}(x))$, and in the case  $p=1$, belong to $C^0(\overline{B_{\varepsilon}}(x))$. If $u\leq v$ a.e. in $\overline{B_{\varepsilon}}(x)$, then 
$\mu_p(\varepsilon,u)\leq \mu_p(\varepsilon,v)$.
\item[(4)] $\mu_p(\varepsilon,u+c)=\mu_p(\varepsilon,u)+c$ for every $c\in \bbR$.
\item[(5)] $\mu_p(\varepsilon,cu)=c\mu_p(\varepsilon,u)$ for every $c\in \bbR$.
\end{itemize}
\end{thm}
The following is \cite[Corollary 2.3]{Mag} in Carnot groups of step $k$:
\begin{cor}\label{cor}
Let $u\in L^p(B_{\varepsilon}(x))$, for $1< p\leq \infty$, or in $C^0(\overline{B_{\varepsilon}}(x))$ for $p=1$. Let $u_{\varepsilon}(z)=u(x\delta_{\varepsilon}(z))$ for $z\in \overline{B_{1}}(0)$, then
\[
\mu_p(\varepsilon,u)(x)=\mu_p(1,u_{\varepsilon})(0).
\]
\end{cor}
\begin{proof}
For every $\lambda\in \bbR$ and $1\leq p<\infty$ it holds:
\[
\|u-\lambda\|^p_{L^p(B_{\varepsilon}(x))}=\int_{B_{\varepsilon}(x)}|u(\xi)-\lambda|^p\, d\xi=\varepsilon^{\sigma_1+\cdots+\sigma_k}\int_{B_{1}(0)}|u_{\varepsilon}(\xi)-\lambda|^p\, d\xi=\varepsilon^{v_1+2v_2+\cdots+kv_k}\|u_{\varepsilon}-\lambda\|^p_{L^p(B_{1}(0))}
\]
and 
\[
\|u-\lambda\|_{L^\infty(B_{\varepsilon}(x))}=\|u_{\varepsilon}-\lambda\|_{L^{\infty}(B_{1}(0))}
\]
and the conclusion follows by (1) in Theorem \ref{prop}. 
\end{proof}
Next we state carefully what is meant by the statement that the asymptotic expansion of the function $u$ in terms of $\mu_p$ holds in the viscosity sense, see~\eqref{minimizer-mu-p} and Definition~\ref{inequal_viscosity_definition}. First, we need the following auxiliary definition.
\begin{defn} Let $h$ be a real valued function defined in a neighborhood of zero.
We say that
$$h(x)\le o(x^2)\text{ as } x\to0^+$$ if any of the three equivalent conditions is satisfied:
\begin{itemize}
\item[a)] $\displaystyle\limsup_{x\to 0^+} \frac{h(x)}{x^2}\le 0$, 
\item[b)] there exists a nonnegative function $g(x)\ge 0$ such that $ h(x)+g(x)=o(x^2) \text{ as } x\to0^+,$
\item[c)] $\displaystyle\lim_{x\to 0^+} \frac{h^+(x)}{x^2}\le 0.$
\end{itemize}
\end{defn}
A similar definition is given for $h(x) \ge o(x^2)\text{ as }x\to0^+$
by reversing the inequalities in a) and c), requiring that $g(x) \le 0$ in b)
and replacing $h^+$ by $h^-$ in c)\footnote{ As usual, we denote by $h^+(x):=\max\{h(x),0\}$ and $h^-(x):=-\min\{h(x),0\}$.}.

Let $f$ and $g$ be two real valued functions defined in a neighborhood of $x_0\in\bbR.$ We say that \emph{$f$ and $g$
are asymptotic functions for $x\to x_0$}, if there exists a function $h$ defined in a neighborhood $V_{x_0}$of $x_0$ such that:
\begin{itemize}
\item[(i)]  $f(x)=g(x)h(x)$ for all $x\in V_{x_0}\setminus\{x_0\}$.
\item[(ii)] $\lim_{x\to x_0}h(x)=1.$
\end{itemize} 

If $f$ and $g$ are asymptotic for $x\to x_0$, then we simply write $f\sim g$ as $x\to x_0.$

\begin{defn}\label{inequal_viscosity_definition}
A continuous function defined in a neighborhood of a point $x\in \bbG$, satisfies
\begin{equation*}
u(x)=\mu_p(\varepsilon,u)(x)+o(\epsilon^2),
\end{equation*}
as $\epsilon\to 0^+$
in the viscosity sense, if the following conditions hold:
\begin{itemize}
\item[(i)] for every continuous function $\phi$ defined in a neighborhood of a point $x$ such that $u-\phi$ has a strict minimum at  $x$  with $u(x)=\phi(x)$ and $\nabla_{V_1} \phi(x)\neq 0$ , we have
$$
\phi(x)\ge\mu_p(\varepsilon,\phi)(x)+ o(\epsilon^2),\quad\hbox{as } \epsilon \to 0^+.
$$
\item[(ii)] for every continuous function $\phi$ defined in a neighborhood of a point $x$ such that $u-\phi$ has a strict maximum at   $x$  with $u(x)=\phi(x)$ and $\nabla_{V_1} \phi(x)\neq 0$, then
$$
\phi(x)\le\mu_p(\varepsilon,\phi)(x)+ o(\epsilon^2),\quad\hbox{as } \epsilon \to 0^+.
$$
\end{itemize}
\end{defn}

\section{The proof of Theorem~\ref{main-thm}}


In order to prove Theorem~\ref{main-thm}, we need the following key lemma.


\begin{lem}[cf. Lemma 3.1 in~\cite{Mag}]\label{lem-main-cc}
Let $\bbG$ be a Carnot group of step $k$. Moreover, let $\Om \subset \bbG$ be an open set and  $x\in \Om$ be a point such that $B_{\varepsilon}(x)\subset \Om$ for all small enough $\ep\leq \ep_0(x)$. Let $1 < p \leq \infty$ and $\xi\in \R^{v_1}\setminus\{0\}$, $\eta \in \R^{v_2}$. Let further $A$ be a symmetric $v_1 \times v_1$ matrix with trace ${\rm tr}(A)$. Moreover, consider the quadratic function  $q:B_{\varepsilon}(x)\to \mathbb{R}$ given by 
\begin{equation}\label{lem-def-q-poly-cc}
q(y)=q(x)+\langle\xi, (x^{-1}y)^{(1)} \rangle_{\R^{v_1}}+ \langle\eta, (x^{-1}y)^{(2)} \rangle_{\R^{v_2}}+\frac{1}{2}\langle A (x^{-1}y)^{(1)}, (x^{-1}y)^{(1)}\rangle_{\R^{v_1}},\quad y\in B_{\varepsilon}(x),
\end{equation}
where $(x^{-1}y)^{(1)}$ and $(x^{-1}y)^{(2)}$ are the horizontal and the vertical components of $x^{-1}y$, respectively and $\langle \cdot,\cdot\rangle_{\bbR^{v_1}}$ and $\langle \cdot,\cdot\rangle_{\bbR^{v_2}}$ denote the Euclidean scalar products on $\bbR^{v_1}$ and $\bbR^{v_2}$, respectively.
It then follows that
\begin{equation}\label{main-thm-assert-cc}
 \mu_p(\varepsilon,q)(x) = q(x) + \varepsilon^2 c \left(  {\rm tr}(A) + (p-2) \frac{\langle A\xi,\xi \rangle_{\bbR^{v_1}}}{|\xi|^2}\right)   + o(\varepsilon^2),
 \end{equation}
where
\[
c:=c(p,v_1,\ldots,v_k)=\frac{ 1}{2 (p+v_1)}  \frac{ \calB \left( \frac{v_k}{2(k-1)!} , \frac{p + \sum_{j=1}^{k-1} j v_j}{2(k-1)!}+1\right) }{\calB \left( \frac{v_k}{2(k-1)!} , \frac{p-2 + \sum_{j=1}^{k-1} j v_j}{2(k-1)!}+1\right)}  \prod_{j=2}^{k-1} \frac{\calB\left( \frac{jv_j}{2k!}, \frac{p +\sum_{i=1}^{j-1} iv_i}{2k!} +1\right)}{\calB\left( \frac{jv_j}{2k!}, \frac{p-2 +\sum_{i=1}^{j-1} iv_i}{2k!} +1\right) }
\] 
and $\calB \left(x,y\right)$ denotes the Beta function $\calB\left(x,y\right)=\int_0^1 t^{x-1}(1-t)^{y-1}\, dt$ for all $x,y>0$.
Furthermore, if $u\in C^2(\Om)$ with $\nabla_{V_1} u(x)\not =0$, then 
\begin{equation}\label{main-thm-assert2-cc}
 \mu_p(\varepsilon, u)(x)=u(x)+c \Delta^N_{p, \G}u(x)\varepsilon^2+o(\varepsilon^2),\quad\hbox{as}\quad \varepsilon\to 0^{+}.
\end{equation} 
\end{lem}

\begin{rem}\label{rem-2.15}
The formula describing the constant $c(p,v_1,\ldots, v_k)$ is complicated and not easily simplified using the properties of the Beta function.
\end{rem}

Before we prove the lemma, let us discuss its assertion in some particular cases:
\begin{ex}[The Euclidean space $\bbR^N$]\label{sect3-ex2}
If $\bbG$ is the Euclidean space $\bbR^N$ then $c(p,v_1,\ldots, v_k)$ agrees with the constant computed in \cite{Mag}, namely  $$c(p,N)=\frac{1}{2(p+N)}.$$
\end{ex}

\begin{ex} [The Heisenberg group $\Hei$, cf.~Example~\ref{ex-heisenberg}]
 If $\G=\Hei$, then quadratic function $q$ in~\eqref{lem-def-q-poly-cc} takes the form:
 \begin{equation*}
q(y)=q(x)+\langle\xi, (x^{-1}y)^{(1)} \rangle+ w (x^{-1}y)^{(2)}+\frac{1}{2}\langle A (x^{-1}y)^{(1)},(x^{-1}y)^{(1)}\rangle_{\bbR^2}, \quad y\in B_{\varepsilon}(x),
\end{equation*}
where $w\in \R$, $\xi\in \R^2\setminus\{0\}$. Furthermore, the constant $c=c(p)$ appearing in \eqref{main-thm-assert-cc} and \eqref{main-thm-assert2-cc} takes the following form
\[
 c(p)=\frac{2}{(p+2)(p+4)}\left( \frac{\Gamma\left( \frac{p+6}{4} \right)}{\Gamma \left( \frac{p+4}{4}\right)} \right)^2,
\]
where for $t>-1$, $\Gamma(t)=\int_0^{\infty} x^{t-1} e^{-x}\, dx$ is the Gamma function. 
\end{ex}

\begin{ex}[Carnot groups of step 2]\label{sect3-ex4}
  Let $\G$ be a Carnot group of step $2$, then the quadratic function $q$ in~\eqref{lem-def-q-poly-cc} takes the form:
 \begin{equation*}
q(y)=q(x)+\langle\xi, (x^{-1}y)^{(1)} \rangle_{\R^n}+ \langle\eta, (x^{-1}y)^{(2)} \rangle_{\R^k}+\frac{1}{2}\langle A (x^{-1}y)^{(1)}, (x^{-1}y)^{(1)}\rangle_{\R^n},\quad y\in B_{\varepsilon}(x),
\end{equation*}
that is $v_1=n, v_2=k$, $\xi\in \R^n\setminus\{0\}$ and $\eta\in \R^k$. Moreover, the constant $c=c(p,n,k)$, appearing in \eqref{main-thm-assert-cc} and \eqref{main-thm-assert2-cc}, takes the following form
\[
 c(p,n,k):= \frac{1}{2(n+p) } \frac{\calB \left( \frac{k}{2}, \frac{n+p+4}{4}\right) }{\calB \left( \frac{k}{2}, \frac{n+p+2}{4}\right) }.
\]
\end{ex}

 In the proof of Lemma~\ref{lem-main-cc} we employ the following integral formula.
\begin{lem}\label{lemdir}
Let $\alpha_1, \ldots, \alpha_n$ be real numbers such that $\alpha_i>-1$ for $i=1,\ldots, n$. It then follows that
\begin{equation}\label{dir}
\int\limits_{T_n} x_1^{\alpha_1} \cdot \ldots \cdot x_n^{\alpha_n} dx = \frac{1}{2^n} \frac{\prod_{i=1}^n  \Gamma( \frac{\alpha_i+1}{2}) }{ \Gamma ( \frac{n+2+\sum \alpha_i }{2})}
\end{equation}
where  $T_n:= \left\{(x_1,\ldots,x_n) \in \R^n: x_1^2 + \ldots + x_n^2 <1,\, x_i \geq 0 \text{ for } i=1,\ldots,n \right\}.$
\end{lem}

\begin{proof}[Proof of Lemma~\ref{lemdir}]
Let $a,b>-1$. Upon applying the change of variables $t=\sin^2 x $, we obtain the following equation:
\[
\int_0^{\frac{\pi}{2}} \sin^a x \cos^b x dx=\int_0^1 t^\frac{a}{2} (1-t)^\frac{b}{2} \frac{1}{2\sqrt{t} \sqrt{1-t}} dt = \frac12 \int_0^1 t^{\frac{a-1}{2}} (1-t)^\frac{b-1}{2} dt = \frac{1}{2} \calB \left( \frac{a+1}{2}, \frac{b+1}{2}\right).
\]

Now we are in a position to calculate the left-hand side of \eqref{dir}. We apply the spherical coordinates 
\[
\begin{cases}
  x_1 &= r \cos \varphi_1\\
  x_2 &= r \sin \varphi_1 \cos \varphi_2  \\
  x_3 &= r \sin \varphi_1 \sin \varphi_2 \cos \varphi_3\\
 \vdots & \vdots\\
 x_{n-1} &= r \sin \varphi_1 \sin \varphi_2 \cdot \ldots \cdot \cos \varphi_{n-1}\\
 x_n &= r \sin \varphi_1 \sin \varphi_2 \ldots \sin \varphi_{n-1}
\end{cases}
\]
with the Jacobian determinant $|J| = r^{n-1} \sin^{n-2} \varphi_1 \sin^{n-3} \varphi_2 \cdot \ldots \cdot \sin \varphi_{n-2}$ and the spherical coordinates varying as follows: $r \in (0,1)$, $\varphi_i \in (0,\pi/2)$ for $i=1,\ldots, n-2$. The result is
\begin{align*}
\int_{T_n}  x_1^{\alpha_1} \cdot \ldots \cdot  x_n^{\alpha_n} dx & = \int_0^1 \int_0^{\frac{\pi}{2}} \!\ldots \int_0^{\frac{\pi}{2}} \bigg[ r^{\sum_{i=1}^n \alpha_i + n-1} \cdot \cos^{\alpha_1} \varphi_1 \sin^{\sum_{i=2}^n \alpha_i +n-2} \varphi_1  \\
& \cdot \cos^{\alpha_2} \varphi_2 \sin^{\sum_{i=3}^n \alpha_i+n-3} \varphi_2 \cdot \ldots \cdot \cos^{\alpha_{n-1}} \varphi_{n-1} \sin^{\alpha_n} \varphi_{n-1} \bigg] d\varphi_1 \ldots d\varphi_{n-1} dr \\
&= \frac{1}{n+\sum_{i=1}^n \alpha_i } \frac12 \calB \left(\frac{\sum_{i=2}^n \alpha_i + n -1}{2}, \frac{\alpha_1 + 1}{2} \right) \frac12 \calB \left(\frac{\sum_{i=3}^n \alpha_i+ n -2}{2}, \frac{\alpha_2 + 1}{2} \right) \\
&  \cdot \ldots \cdot \frac12 \calB \left(\frac{\alpha_n + 1}{2}, \frac{\alpha_{n-1} + 1}{2} \right),
\end{align*}
which is equal to the right-hand side of \eqref{dir} upon using the well-known formula  $\calB \left(x,y\right)=\frac{\Gamma(x)\Gamma(y)}{\Gamma(x+y)}$.
\end{proof}

\begin{proof}[Proof of Lemma~\ref{lem-main-cc}]
In the proof we follow the steps of the proof of Lemma 3.1 in~\cite{Mag}. However, since the setting of Carnot groups differs from the Euclidean one, the computations are to some extent, more demanding and nontrivial.

We begin with computing $\mu_p(\varepsilon,q)$. For $z=(z^{(1)},\ldots,z^{(k)}) \in B:=B(0,1)$, we introduce the following functions:
\[
 q_\varepsilon (z) := q(x\delta_\varepsilon(z)), \quad v_\varepsilon(z) := \frac{q_\varepsilon(z)-q(x)}{\varepsilon} \quad \text{and} \quad v(z):=\langle \xi,(z_1,\ldots,z_{v_1}) \rangle_{\R^n} := \langle \xi , z^{(1)} \rangle_{\R^{v_1}}. 
\]
We know that $\mu_p(\varepsilon,q)(x)=\mu_p(1,q_\varepsilon)(0)$ by Corollary \ref{cor}. Then, by points (4) and (5) of Theorem \ref{prop}, we see that
\[
 \frac{\mu_p(\varepsilon,q)(x) - q(x) }{\varepsilon} = \mu_p(1,v_\varepsilon) (0).
 \]
Let us further observe that 
\begin{equation}\label{eq001}
 v_\varepsilon(z) = \langle \xi,z^{(1)} \rangle + \frac{\varepsilon}{2} \langle A z^{(1)},z^{(1)} \rangle+  \varepsilon \langle \eta, z^{(2)} \rangle  
\end{equation}
which shows that $v_\varepsilon$ converges uniformly to $v$ as $\varepsilon \to 0$ on $\overline{B}$. We appeal to the second part of claim (2) in Theorem \ref{prop} to obtain that $\mu_p(1,v_\varepsilon)(0) \to \mu_p(1,v)(0)$ as $\varepsilon \to 0$. Recall that the characterization of $\lambda=\mu_p(1,v)(0)$ given by \eqref{char} in Theorem \ref{prop} states that if $p\in [1,\infty)$, then $\lambda$ is the unique number such that
\[
 \int_B |\langle \xi , y^{(1)} \rangle -\lambda|^{p-2} (\langle \xi , y^{(1)} \rangle -\lambda) dy = 0.
 \]
On the other hand we have
\[
 \int_B |\langle \xi , y^{(1)} \rangle|^{p-2} (\langle \xi , y^{(1)} \rangle) dy = 0,
 \]
which follows from the symmetry of the unit ball and the following natural change of variables 
$$
\Phi(y^{(1)},y^{(2)},\ldots, y^{(k)}) = (-y^{(1)},y^{(2)},\ldots,y^{(k)}), \quad |J_\Phi|=1, \quad \Phi(B)=B.
$$
It now follows that $\mu_p(1,v)(0)=\lambda = 0$.

If $p=\infty$, then by \eqref{inf}:
\[ \mu_\infty(1,v)(0) = \frac12 \left(\min_{\overline{B}} \langle \xi , y^{(1)} \rangle + \max_{\overline{B}} \langle \xi , y^{(1)} \rangle \right) = \frac12 \left( -|\xi|+|\xi| \right)=0.\]

Next, we split the discussion into the cases depending on the value of $p$.
Let us define 
\[
 \gamma_\varepsilon := \frac{\mu_p(\varepsilon,q)(x)-q(x)}{\varepsilon^2}.
 \]
\subsection{Case 1: $1<p<\infty$.}

For the sake of brevity, we introduce a function $f(s) = |s|^{p-2}s$. Then, upon applying \eqref{char} to $\mu_p(1,v_\varepsilon)(0) = \varepsilon \gamma_\varepsilon$, we obtain
\[
 \int_B f( v_\varepsilon(z) - \varepsilon \gamma_\varepsilon)dz =0.
 \]
By using \eqref{eq001}, this can be transformed to the following expression:
\begin{equation}\label{eq002}
 \int_B f \left( \langle \xi ,z^{(1)} \rangle + \varepsilon \left( \frac{1}{2} \langle A z^{(1)}, z^{(1)} \rangle - \gamma_\varepsilon +\langle \eta, z^{(2)}\rangle \right) \right)dz = 0.
\end{equation}
Without loss of generality we may assume that $|\xi|=1$, since otherwise we can consider the quadratic function $\tilde{q} = q/|\xi|$. Let us apply the change of variables $z=(z^{(1)},z^{(2)},\ldots,z^{(k)}) = (Ry^{(1)},y^{(2)},\ldots,y^{(k)})$ in \eqref{eq002}, where $R$ is a $v_1\times v_1$ rotation matrix with $R^T \xi = e_1$ and $e_1$ denotes the first element of the canonical basis of $\bbR^{\nu_1}$. Set $C = R^T A R$, then \eqref{eq002} reads as
\[
 \int_B f \left(y_1 + \varepsilon \left( \frac{1}{2} \langle C y^{(1)}, y^{(1)} \rangle - \gamma_\varepsilon + \langle \eta, y^{(2)} \rangle \right)\right)dy=0. 
\]
Since $\int_B f(y_1) dy = 0$, it follows that for all $\ep>0$, we have: 
\[  \int_B \frac{1}{\varepsilon} \left( f \left(y_1 + \varepsilon \left( \frac{1}{2} \langle C y^{(1)}, y^{(1)} \rangle - \gamma_\varepsilon + \langle \eta, y^{(2)} \rangle \right)\right) - f(y_1) \right)dy=0. \]
Therefore,  by the Fundamental Theorem of Calculus, we have:
\begin{equation}\label{eq007}
 \int_B \left[ \int_0^1  f' \left( y_1 + t \varepsilon \left( \frac{1}{2} \langle C y^{(1)}, y^{(1)} \rangle - \gamma_\varepsilon + \langle \eta, y^{(2)} \rangle \right) \right) dt \right]  \left( \frac{1}{2} \langle C y^{(1)}, y^{(1)} \rangle - \gamma_\varepsilon + \langle \eta, y^{(2)} \rangle \right) dy =0.
 \end{equation}
Equality \eqref{eq007} implies that $\gamma_\varepsilon$ is a weighted mean value of the function $\frac{1}{2} \langle C y^{(1)}, y^{(1)} \rangle~+~\langle \eta, y^{(2)} \rangle $ over $B$ with respect to a weighted Lebesgue measure $w(y) dy$ for 
$$
w(y):= \int_0^1 f' \left( y_1 + t \varepsilon \left( \frac{1}{2} \langle C y^{(1)}, y^{(1)} \rangle - \gamma_\varepsilon + \langle \eta, y^{(2)} \rangle \right) \right) dt, \quad y\in B.
$$
The weight function $w$ is nonnegative since $f'(s) = (p-1)|s|^{p-2}\geq 0$. Therefore, $\gamma_\varepsilon$ is bounded by $c:=\norm{ \frac{1}{2} \langle C y^{(1)}, y^{(1)} \rangle + \langle \eta, y^{(2)} \rangle }_{L^\infty(B)}$. 

Let us consider any subsequence of $(\gamma_\varepsilon)$ converging to $\gamma_0$ as $\varepsilon \to 0^+$, which for the sake of brevity, we also denote by $(\gamma_\varepsilon)$.  Let us consider two cases. If $2 \leq p < \infty$, then for all $y\in B$ we obtain
\begin{align*}
& \left| \int_0^1 f' \left( y_1 + t \varepsilon \left( \frac{1}{2} \langle C y^{(1)}, y^{(1)} \rangle - \gamma_\varepsilon + \langle \eta, y^{(2)} \rangle \right) \right) dt   \left( \frac{1}{2} \langle C y^{(1)}, y^{(1)} \rangle - \gamma_\varepsilon + \langle \eta, y^{(2)} \rangle \right) \right| \\
& \leq 2c (p-1) \int_0^1  \left| y_1 + t \varepsilon \left( \frac{1}{2} \langle C y^{(1)}, y^{(1)} \rangle - \gamma_\varepsilon + \langle \eta, y^{(2)} \rangle \right) \right|^{p-2} dt  \leq 2c (p-1) (1+ 2c\varepsilon). 
\end{align*}
Therefore, by the dominated convergence theorem the sequence $(\gamma_\varepsilon)$ converges to

\begin{equation}\label{eq003}
 \gamma_0:= \lim_{\varepsilon \to 0} \gamma_\varepsilon = \frac{\int_B |y_1|^{p-2}\left( \frac12 \langle Cy^{(1)},y^{(1)} \rangle  +  \langle \eta, y^{(2)} \rangle \right) dy }{ \int_B |y_1|^{p-2} dy} .
\end{equation}

Let now $1<p<2$. Fix $0< \theta <1$ and split the integral \eqref{eq007} into two parts: over the set $G_\theta :=B \cap \{ |y_1| > \theta \} $ and $F_\theta :=B \cap \{ |y_1| \leq \theta \} $. Observe that for all $y \in G_\theta$ and for all $\varepsilon>0$ satisfying $2c\ep<\theta$, we have the following:
\begin{align*}
& \left| \int_0^1 f' \left( y_1 + t \varepsilon \left( \frac{1}{2} \langle C y^{(1)}, y^{(1)} \rangle - \gamma_\varepsilon + \langle \eta, y^{(2)} \rangle \right) \right) dt   \left( \frac{1}{2} \langle C y^{(1)}, y^{(1)} \rangle - \gamma_\varepsilon + \langle \eta, y^{(2)} \rangle \right) \right| \\
& \leq 2c  \left| |y_1|-2c \varepsilon \right|^{p-2}.
\end{align*}
Moreover,
\begin{equation}\label{eq-conv}
\lim_{\varepsilon \to 0} \int_{G_\theta} \left| |y_1|-2c \varepsilon \right|^{p-2} dy
= \int_{G_\theta} |y_1|^{p-2} dy< \int_{B} |y_1|^{p-2} dy,
\end{equation}
where the inequality holds uniformly for all $\theta \in (0,1)$. Furthermore, the last integral turns out to be finite which can be seen from the explicit calculation below in~\eqref{I}. Hence, by applying Theorem 5.4 in \cite{Mag} to $X= G_\theta$ with $\nu$ being the Lebesgue measure, we obtain the following:
\begin{align*}
& \lim_{\varepsilon \to 0} \int_{G_\theta} \int_0^1 f' \left( y_1 + t \varepsilon \left( \frac{1}{2} \langle C y^{(1)}, y^{(1)} \rangle - \gamma_\varepsilon + \langle \eta, y^{(2)} \rangle \right) \right) dt   \left( \frac{1}{2} \langle C y^{(1)}, y^{(1)} \rangle - \gamma_\varepsilon + \langle \eta, y^{(2)} \rangle \right) dy \\
& = \int_{G_\theta} (p-1) |y_1|^{p-2} \left( \frac{1}{2} \langle C y^{(1)}, y^{(1)} \rangle + \langle \eta, y^{(2)} \rangle - \gamma_0 \right).
\end{align*}
Observe that here the upper bound in \eqref{eq-conv} allows us to conclude that the limit as $\theta \to 0^+$ is finite. We now focus on the part of the integral in \eqref{eq007} involving the set $F_\theta$. Since  $|F_\theta|=\int_{F_{\theta}}1 dy$, then upon writing this integral as in \eqref{I}, one sees that $|F_\theta|=c(k, v_1,\ldots, v_k)\theta$, and so $|F_\theta|\to 0$, as $\theta \to 0^+$. Moreover, it suffices to consider $\theta=2c\ep$ and the related $\int_{F_{2c\ep}} \left| |y_1|-2c \varepsilon \right|^{p-2} dy$. We again appeal to integral \eqref{I} and reduce our computations to finding 
$$
\int_{B_{v_1}(0, R_1)\cap \{ |y_1| \leq 2c\ep \}} \left( 2c \varepsilon-|y_1| \right)^{p-2} dy^{(1)}.
$$
However, direct computation shows that this integral is of order $\ep^{p-1}$, which then allows us to let $\ep\to0^{+}$, and in turn conclude \eqref{eq003}.

In order to approach the proof of~\eqref{main-thm-assert-cc}, we first need to compute integrals in~\eqref{eq003}. We begin with computing the denominator of \eqref{eq003}. Once this is completed, the computation of the numerator will be more straightforward. We write
\begin{align}\label{I}
\begin{split}
I=\int_B |y_1|^{p-2} dy = \int\limits_{\small{B_{v_k}(0,1)}} \int\limits_{\small{B_{v_{k-1}}(0,R_{k-1})} }  \ldots \int\limits_{\small{B_{v_{2}}(0,R_{2})} }  \int\limits_{\small{B_{v_1}(0,R_1)}} |y_1|^{p-2} dy^{(1)} dy^{(2)} \ldots dy^{(k-1)} dy^{(k)},
\end{split}
\end{align}
where for $j=1,\ldots, k$, $B_{v_j}(0, R_j)$ denotes the Euclidean ball in $\R^{v_j}$ centered at $0$ with radius $R_{k}=1$. Furthermore, each radius $R_j>0$ is a function depending on the variables $y^{(i)}$ with $i >j$, with the following property:
\begin{align*}
R_{k-1}=&R_{k-1}(y^{(k)}) = \left(1-\norm{y^{(k)}}^{\frac{2k!}{k}} \right)^{\frac{k-1}{2k!}}\\
R_{k-2}=&R_{k-2}(y^{(k)},y^{(k-1}) = \left( 1- \norm{y^{(k)}}^{\frac{2k!}{k}} - \norm{y^{(k-1)}}^{\frac{2k!}{k-1}} \right)^{\frac{k-2}{2k!}} \\
\vdots & \\
R_{j}= &R_{j}( y^{(k)} \ldots, y^{(j+1)}) = \left(1- \norm{y^{(k)}}^{\frac{2k!}{k}} - \ldots - \norm{y^{(j+1)}}^{\frac{2k!}{j+1}} \right)^{\frac{j}{2k!}} \\
\vdots &\\
R_2  = &R_2(y^{(k)},\ldots, y^{(3)}) = \left(1- \sum_{i=3}^k \norm{y^{(i)}}^{\frac{2k!}{i}}\right)^{\frac{2}{2k!}}\\
R_1  =& R_1 (y^{(k)}, \ldots , y^{(2)}) =  \left(1- \sum_{i=2}^k \norm{y^{(i)}}^{\frac{2k!}{i}}\right)^{\frac{1}{2k!}}.
\end{align*}
Upon applying the scaling change of variables, followed by Lemma \ref{lemdir} with $\alpha_1=p-2$ and $\alpha_i=0$ for $i=2,\ldots, v_1$, we obtain the following equality:
\begin{align}
\int_{B_{v_1}(0,R_1)} |y_1|^{p-2} dy^{(1)} &= R_1^{v_1+p-2} \int_{B_{v_1}(0,1)} |y_1|^{p-2} dy^{(1)} = R_1^{v_1+p-2} 2^{v_1} \int_{T_{v_1}} y_1^{p-2} dy^{(1)} \nonumber \\
&= R_1^{v_1+p-2} \frac{\Gamma \left( \frac{p-1}{2} \right) \Gamma \left(\frac{1}{2} \right)^{v_1-1}}{\Gamma \left( \frac{v_1+p}{2} \right)}. \label{eq004}
\end{align}
Using \eqref{eq004} in $I$, we see that
\begin{equation}\label{I2}
I =\frac{\Gamma \left( \frac{p-1}{2} \right) \Gamma \left(\frac{1}{2} \right)^{v_1-1}}{\Gamma \left( \frac{v_1+p}{2} \right)} \int_{B_{v_k}(0,1)} \ldots \int_{B_{v_2}(0,R_2)} R_1^{v_1+p-2} dy^{(2)} \ldots dy^{(k)}.
\end{equation}
Since $R_1^{v_1+p-2}$ is a radial function with respect to $y^{(2)},\ldots, y^{(k)}$, in particular with respect to $y^{(2)}$, we use the spherical coordinates together with the observation that $R_1 = \left(R_2^{\frac{2k!}{2}} - \norm{y^{(2)}}^{\frac{2k!}{2}} \right)^{\frac{1}{2k!}}$ to obtain the following:
\begin{align*}
\int_{B_{v_2}(0,R_2)} R_1^{v_1+p-2} dy^{(2)} &= \frac{2 \sqrt{\pi}^{v_2}}{\Gamma \left( \frac{v_2}{2}\right) } \int_0^{R_2} \left(R_2^{\frac{2k!}{2}} - r^{\frac{2k!}{2}} \right)^{\frac{v_1+p-2}{2k!}} r^{v_2-1} dr\\
& =  \frac{2 \sqrt{\pi}^{v_2}}{\Gamma \left( \frac{v_2}{2}\right) } \int_0^1 R_2^{\frac{v_1+p-2}{2}} (1-s^{\frac{2k!}{2}})^{\frac{v_1+p-2}{2k!}} R_2^{v_2-1} s^{v_2-1 } R_2 ds  \qquad (R_2 s:= r)  \\
& = \frac{2 \sqrt{\pi}^{v_2}}{\Gamma \left( \frac{v_2}{2}\right) } R_2^{\frac{2v_2+v_1+p-2}{2}} \int_0^1 (1-s^{\frac{2k!}{2}})^{\frac{v_1+p-2}{2k!}} s^{v_2-1} ds\\
& = \frac{2 \sqrt{\pi}^{v_2}}{\Gamma \left( \frac{v_2}{2}\right) } R_2^{\frac{2v_2+v_1+p-2}{2}} \frac{2}{2k!} \int_0^1 (1-t)^{\frac{v_1+p-2}{2k!}} t^{\frac{2(v_2-1)}{2k!}} t^{\frac{2}{2k!}-1} dt  \qquad (t:=s^{\frac{2k!}{2}}) \\
&= \frac{2 \sqrt{\pi}^{v_2}}{\Gamma \left( \frac{v_2}{2}\right) } R_2^{\frac{2v_2+v_1+p-2}{2}} \frac{2}{2k!} \int_0^1 (1-t)^{\frac{v_1+p-2}{2k!}} t^{\frac{2v_2}{2k!} -1}  dt \\
&= \frac{4 \sqrt{\pi}^{v_2}}{ 2k! \Gamma \left( \frac{v_2}{2}\right) } R_2^{\frac{2v_2+v_1+p-2}{2}} \calB \left( \frac{2v_2}{2k!}, \frac{v_1+2k!+p-2}{2k!} \right).
\end{align*}
In summarise, we now have
\[ I= \frac{4 \Gamma \left( \frac{p-1}{2} \right) \sqrt{\pi}^{v_1+v_2-1}}{2k!\Gamma \left( \frac{v_1+p}{2} \right) \Gamma \left( \frac{v_2}{2}\right)} \calB\left( \frac{2v_2}{2k!}, \frac{v_1+2k!+p-2}{2k!} \right) \int_{B_{v_k}(0,1)} \ldots \int_{B_{v_3}(0,R_3)} R_2^{\frac{2v_2+v_1+p-2}{2}} dy^{(3)} \ldots dy^{(k)}.\]
In order to complete the computation of the iterated integral $I$, we need to proceed similarly to the previous case. As it turns out, the key step is to calculate the following integral:
\begin{equation}\label{I4}
\int_{B_{v_j} (0, R_j)} R_{j-1}^{\theta_j} dy^{(j)} 
\end{equation}
where $\theta_j >0$ is defined inductively for $j=2,3,\ldots,k-1$. From the previous computations we see that $\theta_2 = v_1+p-2$ and $\theta_3= \frac{2v_2+v_1+p-2}{2}$.

Let us observe, that from the construction of $R_j$, it follows that
\[ R_{j-1} = \left(R_j^{\frac{2k!}{j}} - \norm{y^{(j)}}^{\frac{2k!}{j}} \right)^{\frac{j-1}{2k!}}.\]
Hence
\begin{align*}
\int \limits_{B_{v_j} (0, R_j)} \!\!\!R_{j-1}^{\theta_j} dy^{(j)}  &= \!\!\! \int \limits_{B_{v_j}(0,R_j)} \!\!\! \left(R_j^{\frac{2k!}{j}} - \norm{y^{(j)}}^{\frac{2k!}{j}} \right)^{\frac{(j-1)\theta_j}{2k!}} \!\! dy^{(j)} = \frac{2 \sqrt{\pi}^{v_j}}{\Gamma \left(\frac{v_j}{2} \right)} \int_0^{R_j} \!\! \left(R_j^{\frac{2k!}{j}} - r^{\frac{2k!}{j}} \right)^{\frac{(j-1)\theta_j}{2k!}} \!\! r^{v_j-1}dr,
\end{align*}
which again follows by the integrand being radial. We apply the change of variables $R_j s:=r$ to obtain
\begin{align*}
\int_0^{R_j} \left(R_j^{\frac{2k!}{j}} - r^{\frac{2k!}{j}} \right)^{\frac{(j-1)\theta_j}{2k!}} r^{v_j-1}dr &= \int_0^1 \left(R_j^{\frac{2k!}{j}} - R_j^{\frac{2k!}{j}}s^{\frac{2k!}{j}} \right)^{\frac{(j-1)\theta_j}{2k!}} R_j^{v_j-1} s^{v_j-1} R_j ds \\
&= R_j^{\frac{(j-1)\theta_j + j v_j}{j}} \int_0^1 (1-s^{\frac{2k!}{j}})^{\frac{(j-1)\theta_j}{2k!}} s^{v_j-1} ds \\
&= R_j^{\frac{(j-1)\theta_j + j v_j}{j}} \int_0^1 (1-t)^{\frac{(j-1)\theta_j}{2k!}} t^{\frac{j(v_j-1)}{2k!}} \frac{j}{2k!} t^{\frac{j-2k!}{2k!}} dt \qquad (t:= s^{\frac{2k!}{j}}) \\
&= \frac{j}{2k!} R_j^{\frac{(j-1)\theta_j + j v_j}{j}} \int_0^1 (1-t)^{\frac{(j-1)\theta_j}{2k!}} t^{\frac{jv_j -2k!}{2k!}} dt \\
&=\frac{j}{2k!} R_j^{\frac{(j-1)\theta_j + j v_j}{j}} \calB \left( \frac{jv_j}{2k!}, \frac{(j-1)\theta_j}{2k!} +1\right).
\end{align*} 
Therefore $\theta_j$ is defined by the following recursive formula 
\[\theta_2 = v_1+p-2  \qquad \text{and}  \qquad \theta_{j+1} = v_j + \frac{j-1}{j} \theta_j, \quad  j=2,\dots,k-1,\]
which leads to the following explicit formula:
\begin{equation}\label{teta-key-lemma}
 \theta_{j+1} = \frac{p-2 +\sum_{i=1}^j iv_i}{j} .
 \end{equation}
Indeed, observe that
\[\frac{j-1}{j} \cdot \frac{p-2 +\sum_{i=1}^{j-1} iv_i}{j-1} + v_j = \frac{p-2 +\sum_{i=1}^j iv_i}{j}.\]
Now we are in a position to complete the calculation of the integral $I$, cf. \eqref{I} and \eqref{I2}:
\begin{align*}
I &= \frac{\Gamma \left( \frac{p-1}{2} \right) \sqrt{\pi}^{v_1-1}}{\Gamma \left( \frac{v_1+p}{2} \right)} \int\limits_{B_{v_k}(0,1)} \ldots \int\limits_{B_{v_2}(0,R_2)} R_1^{v_1+p-2} dy^{(2)} \ldots dy^{(k)} \\
&= \frac{\Gamma \left( \frac{p-1}{2} \right) \sqrt{\pi}^{v_1-1}}{\Gamma \left( \frac{v_1+p}{2} \right)} \frac{4 \sqrt{\pi}^{v_2}}{ 2k! \Gamma \left( \frac{v_2}{2}\right) }  \calB \left( \frac{2v_2}{2k!}, \frac{v_1+2k!+p-2}{2k!} \right)\!\! \int\limits_{B_{v_k}(0,1)}\!\! \ldots\!\! \int\limits_{B_{v_3}(0,R_3)}\!\!\!\! R_2^{\frac{2v_2+v_1+p-2}{2}} dy^{(3)} \ldots dy^{(k)}.
\end{align*}
Each inner integral of $R_{j-1}^{\theta_j}$ gives rise to the multiplicative constant 
\[ \frac{ \sqrt{\pi}^{v_j}}{\Gamma \left(\frac{v_j}{2} \right)} \frac{j}{k!} \calB \left( \frac{jv_j}{2k!}, \frac{(j-1)\theta_j}{2k!} +1\right) \]
in the value of the iterated integral. Therefore, we end up with
\begin{align*}
I&=\frac{\Gamma \left( \frac{p-1}{2}\right) \sqrt{\pi}^{-1 +\sum_{j=1}^{k-1} v_j}  (k-1)!}{(k!)^{k-1}\Gamma \left( \frac{v_1+p}{2}\right) \prod_{j=2}^{k-1} \Gamma \left( \frac{v_j}{2}\right)} \prod_{j=2}^{k-1} \calB\left( \frac{jv_j}{2k!}, \frac{(j-1)\theta_j}{2k!} +1\right) \int_{B_{v_k}(0,1)} R_{k-1}^{\theta_k} dy^{(k)}.
\end{align*}
Recall, that $\theta_k= \frac{p-2 + \sum_{j=1}^{k-1} j v_j}{k-1}$, $R_k = (1-\norm{y^{(k)}}^{\frac{2k!}{k}})^{\frac{k-1}{2k!}}$ and compute
\begin{align*} 
\int_{B_{v_k}(0,1)} R_{k-1}^{\theta_k} dy^{(k)} &=\int_{B_{v_k}(0,1)} (1-\norm{y^{(k)}}^{\frac{2k!}{k}})^{\frac{\theta_k (k-1)}{2k!}} dy^{(k)} \\
&= \frac{2 \sqrt{\pi}^{v_k}}{\Gamma \left( \frac{v_k}{2} \right)} \int_0^1 (1-r^{\frac{2k!}{k}})^{\frac{\theta_k (k-1)}{2k!}} r^{v_k-1} dr \qquad (s:=r^{\frac{2k!}{k}}) \\
&= \frac{2 \sqrt{\pi}^{v_k}}{\Gamma \left( \frac{v_k}{2} \right)} \frac{1}{2(k-1)!} \int_0^1 (1-s)^{\frac{\theta_k(k-1)}{2k!}} s^{\frac{v_k-1}{2(k-1)!}} s^{\frac{1}{2(k-1)!}-1} ds \\
&= \frac{\sqrt{\pi}^{v_k}}{\Gamma \left( \frac{v_k}{2} \right) (k-1)!} \int_0^1 (1-s)^{\frac{\theta_k(k-1)}{2k!}} s^{\frac{v_k-2(k-1)!}{2(k-1)!}} ds \\
&= \frac{\sqrt{\pi}^{v_k}}{\Gamma \left( \frac{v_k}{2} \right) (k-1)!}  \calB \left( \frac{v_k}{2(k-1)!} , \frac{\theta_k (k-1)}{2(k-1)!}+1\right).
\end{align*}
Hence we arrive at
\begin{equation}\label{I3}
I = \frac{\Gamma \left( \frac{p-1}{2}\right) \sqrt{\pi}^{-1 +\sum_{i=1}^{k} v_i} }{(k!)^{k-1}\Gamma \left( \frac{v_1+p}{2}\right) \prod_{i=2}^{k} \Gamma \left( \frac{v_i}{2}\right)} \calB \left( \frac{v_k}{2(k-1)!} , \frac{\theta_k }{2(k-2)!}+1\right) \prod_{i=2}^{k-1} \calB\left( \frac{jv_j}{2k!}, \frac{(j-1)\theta_j}{2k!} +1\right). 
\end{equation}

Next we consider the integral in the numerator of \eqref{eq003}, namely
\[
 J:= \int_B |y_1|^{p-2} \left(  \frac12 \langle Cy^{(1)},y^{(1)} \rangle + \langle \eta, y^{(2)} \rangle \right) dy. 
 \]
We note that $\int_B  \langle \eta, y^{(2)}\rangle |y_1|^{p-2} = 0$, which follows by applying the change of variables $$\psi(y^{(1)}, y^{(2)},y^{(3)},\ldots, y^{(k)})=(y^{(1)}, - y^{(2)},y^{(3)},\ldots, y^{(k)}),$$ with $|J\psi|=1$ and $\psi(B)=B$, resulting in the value of the integral being invariant under multiplication by $-1$. Let us denote the coefficients of matrix $C$ as follows: $C=[c_{ij}]_{i,j=1,\ldots,v_1}$, then
\[
 2J = \underbrace{c_{11} \int_B |y_1|^{p}dy}_{J_1} + \underbrace{\sum_{i\ne j} c_{ij} \int_B |y_1|^{p-2} y_i y_j dy}_{J_2} + \underbrace{\sum_{i=2}^{v_1} c_{ii}\int_B |y_1|^{p-2}y_i^2 dy}_{J_3}.
 \]

Observe, that by the symmetry of $B$, every integral term of the sum $J_2$ vanishes. We will handle $J_1$ and $J_3$ analogously to $I$. First, for $i=2,\ldots,v_1$ we compute the following integrals 
\begin{align}
\int_{B_{v_1}(0,R_1)} |y_1|^{p-2}y_i^2 dy^{(1)} &=R_1^{v_1+p} \frac{\Gamma \left( \frac{p-1}{2}\right) \Gamma \left( \frac{3}{2}\right) \Gamma \left(\frac{1}{2} \right)^{v_1-2}}{\Gamma \left( \frac{p+v_1+2}{2}\right)} = R_1^{v_1+p}\frac{\sqrt{\pi}^{v_1-1}\Gamma \left(\frac{p-1}{2}\right)}{2 \Gamma \left( \frac{p+v_1+2}{2}\right)}, \label{key-lem-eq6c}
\end{align}
where again we use Lemma~\ref{lemdir} and the familiar property $\Gamma(1+s)=s\Gamma(s)$ with $s=\frac12$ (cf. computations at~\eqref{eq004}).

Notice that the calculations summarised in \eqref{I3} work for an arbitrary $p>1$. More precisely, the integrals $J_1$ and $J_3$ over the ball $B$, can be expressed in the same way as in \eqref{I}, the multiplicative constants arising from the computation of integrals~\eqref{I4} will be the same but with the exponents $\theta_j$ replaced by the exponents $\theta_j'$ defined by the following formula (cf. definition of $\theta_{j}$ in~\eqref{teta-key-lemma}):
\[
 \theta'_j = \frac{p+ \sum_{i=1}^{j-1} i v_i}{j-1}.
 \] 
Therefore, by using~\eqref{key-lem-eq6c} and calculations analogous to those between formula~\eqref{I4} and~\eqref{I3} we arrive at
\begin{align*}
J_3 &= \sum_{i=2}^{v_1} c_{ii} \frac{\sqrt{\pi}^{-1 + \sum_{j=1}^{k-1}v_j}\Gamma \left(\frac{p-1}{2}\right) (k-1)!}{2 (k!)^{k-1}\Gamma \left( \frac{p+v_1+2}{2}\right) \prod_{j=2}^{k-1} \Gamma \left( \frac{v_j}{2}\right)} \prod_{j=2}^{k-1} \calB\left( \frac{jv_j}{2k!}, \frac{(j-1)\theta'_j}{2k!} +1\right) \!\int \limits_{B_{v_k}(0,1)} \!(1-\norm{y^{(k)}}^{\frac{2k!}{k}})^{\frac{\theta'_k (k-1)}{2k!}} dy^{(k)} \\
& =\sum_{i=2}^{v_1} c_{ii} \frac{\sqrt{\pi}^{-1 + \sum_{j=1}^{k}v_j}\Gamma \left(\frac{p-1}{2}\right) }{2 (k!)^{k-1}\Gamma \left( \frac{p+v_1+2}{2}\right) \prod_{j=2}^{k} \Gamma \left( \frac{v_j}{2}\right)} \calB \left(\frac{v_k}{2(k-1)!},\frac{\theta_k'}{2(k-2)!}+1 \right) \prod_{j=2}^{k-1} \calB\left( \frac{jv_j}{2k!}, \frac{(j-1)\theta'_j}{2k!} +1\right). 
\end{align*}
Moreover, in order to compute $J_1$, we proceed computationally the same way we did for  for \eqref{I} with the power $p$ instead of $p-2$, and obtain~\eqref{I3} with $p$ now corresponding to $p+2$:
\[
J_1 = c_{11} \frac{\Gamma \left( \frac{p+1}{2}\right) \sqrt{\pi}^{-1 +\sum_{j=1}^{k} v_j} }{(k!)^{k-1}\Gamma \left( \frac{v_1+p+2}{2}\right) \prod_{j=2}^{k} \Gamma \left( \frac{v_j}{2}\right)} \calB \left( \frac{v_k}{2(k-1)!} , \frac{\theta'_k }{2(k-2)!}+1\right) \prod_{j=2}^{k-1} \calB\left( \frac{jv_j}{2k!}, \frac{(j-1)\theta'_j}{2k!} +1\right). 
\]
We collect the above calculations to arrive at
\begin{align*}
J =\frac{J_1+J_3}{2} &= \frac{ \sqrt{\pi}^{-1 +\sum_{j=1}^{k} v_j} }{2(k!)^{k-1}\Gamma \left( \frac{v_1+p+2}{2}\right) \prod_{j=2}^{k} \Gamma \left( \frac{v_j}{2}\right)} \calB \left( \frac{v_k}{2(k-1)!} , \frac{\theta'_k }{2(k-2)!}+1\right)  \\
& \times \left( c_{11}\Gamma \left( \frac{p+1}{2}\right) + \sum_{i=1}^{v_1} \frac{1}{2} c_{ii} \Gamma \left( \frac{p-1}{2}\right)  \right)\prod_{j=2}^{k-1} \calB\left( \frac{jv_j}{2k!}, \frac{(j-1)\theta'_j}{2k!} +1\right)  \\
&= \frac{\Gamma \left( \frac{p-1}{2}\right) \sqrt{\pi}^{-1 +\sum_{j=1}^{k} v_j} }{4(k!)^{k-1}\Gamma \left( \frac{v_1+p+2}{2}\right) \prod_{j=2}^{k} \Gamma \left( \frac{v_j}{2}\right)} \calB \left( \frac{v_k}{2(k-1)!} , \frac{\theta'_k }{2(k-2)!}+1\right) \\
& \times  \left( c_{11} (p-1) + \sum_{i=2}^{v_1} c_{ii} \right)  \prod_{j=2}^{k-1} \calB\left( \frac{jv_j}{2k!}, \frac{(j-1)\theta'_j}{2k!} +1\right),
\end{align*}
where we again use the familiar property of the $\Gamma$ function as in~\eqref{key-lem-eq6c}. It now follows that  
\begin{align*}
\gamma_0 & =\frac{J}{I}= \frac{ \Gamma \left( \frac{p+v_1}{2}\right)}{4 \Gamma \left( \frac{p+2+v_1}{2}\right)}  \frac{ \calB \left( \frac{v_k}{2(k-1)!} , \frac{\theta'_k }{2(k-2)!}+1\right) }{\calB \left( \frac{v_k}{2(k-1)!} , \frac{\theta_k }{2(k-2)!}+1\right)}  \left( c_{11} (p-1) + \sum_{i=2}^{v_1} c_{ii} \right)\prod_{j=2}^{k-1} \frac{\calB\left( \frac{jv_j}{2k!}, \frac{(j-1)\theta'_j}{2k!} +1\right)}{\calB\left( \frac{jv_j}{2k!}, \frac{(j-1)\theta_j}{2k!} +1\right) } \nonumber\\
&= \frac{ 1}{2 (p+v_1)}  \frac{ \calB \left( \frac{v_k}{2(k-1)!} , \frac{\theta'_k }{2(k-2)!}+1\right) }{\calB \left( \frac{v_k}{2(k-1)!} , \frac{\theta_k }{2(k-2)!}+1\right)}  \prod_{j=2}^{k-1} \frac{\calB\left( \frac{jv_j}{2k!}, \frac{(j-1)\theta'_j}{2k!} +1\right)}{\calB\left( \frac{jv_j}{2k!}, \frac{(j-1)\theta_j}{2k!} +1\right) } \left( c_{11} (p-1) + \sum_{i=2}^{v_1} c_{ii} \right) \nonumber \\
& = c(p,v_1,\ldots,v_k) \cdot \left( c_{11} (p-1) + \sum_{i=2}^{v_1} c_{ii} \right),
\end{align*}
where the constant $c(p,v_1,\ldots,v_k)$ is defined with the above equality (see also Remark~\ref{rem-2.15} and Examples~\ref{sect3-ex2}-\ref{sect3-ex4} in Section 3 for further discussion about this constant).

In order to arrive at assertion \eqref{main-thm-assert-cc}, we express the constants $c_{11}$ and ${\rm tr} (C)$ in terms of the matrix $A$ and the vector $\xi$. Recall that $C=R^T A R$ and $R^T \xi = e_1$, which imply that
\[
c_{11} = \langle C e_1,e_1 \rangle = \langle C R^T \xi , R^T \xi \rangle = \langle R (R^T A R) R^T \xi,\xi \rangle = \langle A \xi ,\xi \rangle,
\]
moreover, the orthogonality of $R$ implies that ${\rm tr} (C)= {\rm tr} (R^T A R) = {\rm tr} (A)$. Therefore, we can conclude that
\[
\gamma_0 = c(p,v_1,\ldots,v_k) ( \langle A \xi, \xi \rangle (p-2) + {\rm tr} (A)),
\]
which upon substituting $\xi$ with $\xi/|\xi|$, proves the assertion \eqref{main-thm-assert-cc}.

We now consider the second assertion of the lemma, namely the asymptotic formula~\eqref{main-thm-assert2-cc} for $\mu_p(\ep,u)$ and $u\in C^2(\Om)$. 
  Suppose $\ep>0$ is chosen so that $\overline{B_{\ep}(x)}\subset \Om$. Consider the function $q(y)$ as in \eqref{lem-def-q-poly-cc},  with
 \[
  q(x)=u(x),\quad \xi=\nabla_{V_1} u(x),\quad A=\nabla_{V_1}^2 u(x),\quad {\rm and} \quad \eta =2 \nabla_{V_2} u (x).
 \]
Notice that with this notation (and by the assumption $\xi\not=0$), it holds that
\[
  \Delta^N_{p, \G}u(x)=\text{tr}(A)+(p-2) \frac{\langle A\xi,\xi \rangle}{|\xi|^2}.
\]
 Set $u_{\ep}(z)=u(x\delta_\varepsilon(z))$ and $q_{\ep}(z)=q(x\delta_\varepsilon(z))$. Since $u\in C^2(\Om)$, it follows that for all $t>0$, there exists $\ep(t)>0$ such that for every $z\in \overline{B}$ and all $\ep\in(0,\ep(t))$ it holds $|u_{\ep}(z)-q_{\ep}(z)|<t \ep^2$.  Furthermore, by claims (4) and (5) of Theorem~\ref{prop} we have $\mu_p(\ep, q\pm t \ep^2)(x)=\mu_p(\ep, q)(x)\pm t \ep^2$. These observations together with Corollary~\ref{cor} and Part (3) of Theorem~\ref{prop} allow us to obtain the following estimates:
\[
\frac{ \mu_p(\ep, q)-u(x)}{\ep^2}-t \leq \frac{ \mu_p(\ep, u)-u(x)}{\ep^2} \leq \frac{ \mu_p(\ep, q)-u(x)}{\ep^2}+t.
\]
Applying \eqref{main-thm-assert-cc} we obtain
\begin{align*}
c(p,v_1,\ldots,v_k)  \Delta^N_{p, \G}u(x)-t &\leq \liminf_{\ep\to 0}  \frac{ \mu_p(\ep, q)-u(x)}{\ep^2} \nonumber \\
&\leq \limsup_{\ep \to 0}\frac{ \mu_p(\ep, q)-u(x)}{\ep^2} \leq  c(p,v_1,\ldots,v_k)  \Delta^N_{p, \G}u(x)+t,
 \end{align*}
which implies the assertion~\eqref{main-thm-assert2-cc} for $1<p<\infty$.

\subsection{Case 2: $p=\infty$.}

We need to demonstrate that the expression
\begin{align*}
 \gamma_\varepsilon &= \frac{\mu_\infty (\varepsilon,q)-q(x)}{\varepsilon^2} \nonumber \\
 & = \frac{1}{2\varepsilon} \bigg( \min_{y \in \overline{B}} \left[ \langle \xi ,y^{(1)} \rangle + \varepsilon \left( \langle \eta , y^{(2)} \rangle + \frac12 \langle Ay^{(1)},y^{(1)} \rangle \right) \right] \nonumber  \\
 & + \max_{y \in \overline{B}} \left[ \langle \xi ,y^{(1)} \rangle + \varepsilon \left( \langle \eta , y^{(2)} \rangle + \frac12 \langle Ay^{(1)},y^{(1)} \rangle \right) \right] \bigg)
\end{align*}
has a limit as $\ep \to 0$.

Let us define a function $g: \bbG \rightarrow \mathbb{R}$ by setting $g(y) = \langle \xi ,y^{(1)} \rangle + \langle \eta, y^{(2)} \rangle + \frac12 \langle Ay^{(1)},y^{(1)} \rangle $. Observe further, that the change of variables  $ y =\delta_{1/\varepsilon} (z)$ implies the following equalities:
\[
 \min_{y \in  \overline{B_1(0)}} \left[ \langle \xi ,y^{(1)} \rangle + \varepsilon \left( \langle \eta , y^{(2)} \rangle + \frac12 \langle Ay^{(1)},y^{(1)} \rangle \right) \right] = \frac{1}{\varepsilon} \min_{z \in \overline{B_\varepsilon(0)}} g(z),
  \]
and
\[
 \max_{y \in  \overline{B_1(0)}} \left[ \langle \xi ,y^{(1)} \rangle + \varepsilon \left( \langle \eta , y^{(2)} \rangle + \frac12 \langle Ay^{(1)},y^{(1)} \rangle \right) \right] = \frac{1}{\varepsilon} \max_{z \in \overline{B_\varepsilon(0)}} g(z), 
 \]
and it follows that
$$ \gamma_\ep=\frac{1}{2\varepsilon^2} \left( \min_{z \in \overline{B_\varepsilon(0)}} g(z) + \max_{z \in \overline{B_\varepsilon(0)}} g(z) \right).$$
 
Next we note that $\nabla_{V_1} g(0) = \xi \ne 0$, thus we can apply Lemma 1.5 and 1.6 in \cite{FP}, and affirm that for all small enough $\varepsilon$, there exist points $P_{\varepsilon,M}=(y^{(1)}_{\varepsilon,M},\ldots,y^{(k)}_{\varepsilon,M})$ and $P_{\varepsilon,m}=(y^{(1)}_{\varepsilon,m},\ldots, y^{(k)}_{\varepsilon,m})$ in $\partial B_\varepsilon(0)$ with the following properties:
\[
 \max_{\overline{B_\varepsilon(0)}} g  = g(P_{\varepsilon,M}) \qquad {\rm and} \qquad  \min_{\overline{B_\varepsilon(0)}} g  = g(P_{\varepsilon,m}).
 \]In terms of the expression we have the following estimate 
\begin{equation}\label{e0224}
\frac{1}{2\varepsilon^2}  \left( g(P_{\varepsilon,m}) + g(-P_{\varepsilon,m}) \right) \, \leq \, \,  \gamma_\ep \, \,  \leq  \, \frac{1}{2\varepsilon^2} \left( g(P_{\varepsilon,M}) + g(-P_{\varepsilon,M}) \right).
\end{equation} 
Moreover, by applying again \cite[Lemma 1.6]{FP}, we have that
\[
\lim_{\varepsilon \to 0} \frac{y^{(1)}_{\varepsilon,M}}{\varepsilon} = \frac{\xi}{|\xi|} \qquad {\rm and} \qquad  \lim_{\varepsilon \to 0}  \frac{y^{(1)}_{\varepsilon,m}}{\varepsilon} = -\frac{\xi}{|\xi|}, 
\] 
which implies
\begin{align*}
\frac{1}{2 \varepsilon^2} ( g(P_{\varepsilon,M}) + g(-P_{\varepsilon,M})) &= \frac{1}{4\varepsilon^2} \left( \langle A y^{(1)}_{\varepsilon,M},y^{(1)}_{\varepsilon,M} \rangle + \langle A -y^{(1)}_{\varepsilon,M},-y^{(1)}_{\varepsilon,M}\rangle \right) \\
& = \frac{1}{2} \langle A \frac{y^{(1)}_{\varepsilon,M}}{\varepsilon},\frac{y^{(1)}_{\varepsilon,M}}{\varepsilon}\rangle \xrightarrow{\varepsilon \to 0} \frac{1}{2} \frac{\langle A \xi , \xi \rangle }{|\xi|^2}.
\end{align*}
We treat the left-hand side of \eqref{e0224} similarly to conclude that
\[
 \mu_\infty(\varepsilon,q)=q(x) + \frac{\varepsilon^2}{2} \frac{\langle A \xi , \xi \rangle}{|\xi|^2} + o(\varepsilon^2).
\]
Upon repeating the reasoning similar to the one for $\Delta^N_{p, \G}$, we obtain that asymptotic formula~\eqref{main-thm-assert2-cc} holds for $\Delta^N_{\infty, \G}$ as well. Thus, the proof of Lemma~\ref{lem-main-cc} is completed for all $1<p\leq \infty$.
\end{proof}

We are now in  position to prove Theorem~\ref{main-thm}.

\begin{proof}[The proof of Theorem~\ref{main-thm}]
  Let $B(x)\subset \Om$ be ball and let us fix $u \in C^0(\Omega)$ and $\phi\in C^2(B(x))$ with $\nabla_{V_1}\phi(x)\not=0$. The asymptotic formula~\eqref{main-thm-assert2-cc} implies that
  \begin{equation}\label{e25}
   \phi(x)=\mu_p(\varepsilon, \phi)(x)-c(p, v_1,\ldots, v_k) \Delta_{p, \bbG}^N \phi(x)\varepsilon^2+o(\varepsilon^2),\quad\hbox{as}\quad \varepsilon\to 0.
  \end{equation}
Suppose that $u$ is a viscosity solution, in the sense of Definition~\ref{def25}, to the equation $\Delta_{p, \bbG}^N u = 0$ in $\Omega$. Thus, in particular, $u$ satisfies parts (i) and (ii) of Definition~\ref{def25}. Since $u$ is a viscosity supersolution of $\Delta_{p, \bbG}^N=0$ in $\Om$, then at point $x$, for $\phi$ as above such that $u - \phi$ has a strict minimum at $x$ and $u(x)=\phi(x)$, it holds that $\Delta_{p, \bbG}^N \phi(x) \leq 0$. Therefore, from \eqref{e25} we obtain
\[
 \phi(x) \geq \mu_p(\varepsilon, u)(x)+o(\varepsilon^2),\quad\hbox{as}\quad \varepsilon\to 0, 
\]
which proves that $\phi$ at $x$ satisfies part (i) of Definition~\ref{inequal_viscosity_definition}. By using the fact that $u$ is also a viscosity subsolution (and so $u$ satisfies part (ii) of Definition~\ref{def25}) we show that inequality in part (ii) of Definition \ref{inequal_viscosity_definition} holds as well. This proves that $u(x)=\mu_p(\varepsilon,u)(x)+o(\varepsilon^2)$ as $\varepsilon \to 0$ in the viscosity sense.

Now we will prove the converse. Suppose, that $u(x)=\mu_p(\varepsilon,u)(x)+ o(\varepsilon^2)$ as $\varepsilon \to 0$ in the viscosity sense. If $u-\phi$ attains a strict minimum at $x$, then by Definition~\ref{inequal_viscosity_definition}, it follows that
 $\phi(x) \geq \mu_p(\varepsilon,\phi)(x) + o(\varepsilon^2)$ as $\varepsilon \to 0$. Using this result in \eqref{e25}, we get
\[
 \Delta_{p, \bbG}^N \phi (x)  = \frac{\mu_p(\varepsilon,\phi)(x) - \phi(x) }{c(p, v_1,\ldots, v_k) \varepsilon^2} + o(1) \leq o(1),
 \]
as $\varepsilon \to 0$, and hence $\Delta_{p, \bbG}^N \phi(x) \leq 0$. We apply a similar reasoning in the case $u-\phi$ has a strict maximum at $x$. This proves, that $u$ is a viscosity solution of $\Delta_{p, \bbG}^N u=0$ in $\Omega$.
\end{proof}

We close this section with a remark of Theorem~\ref{main-thm} in the case $p=1$. 

\begin{rem}\label{ex-fail}
  The techniques used in the proof of \cite[Lemma 3.1]{Mag} cannot be easily adapted to obtain Theorem~\ref{main-thm} for $p=1$. Indeed, the Implicit Function Theorem employed on pg. 11 in~\cite{Mag} for an ellipsoid in $\R^n$ and $f_{\ep}$, cannot be used directly already in the setting of the Heisenberg group $\Hei$. The noncommutativity of the group operation in $\Hei$ together with the formula for the Koranyi--Reimann distance result in the singular set within the ellipsoid in $\Hei$ and prevent us from using the Implicit Function Theorem. The alternative approaches lead to difficulties of computational nature.
\end{rem}

\textbf{Acknowledgements:}
T. Adamowicz and B. Warhurst were supported by a grant of National Science Centre, UMO-2017/25/B/ST1/01955. A. Pinamonti is partially supported by the University of Trento and GNAMPA of INDAM.
The authors would like to thank Professor Rolando Magnanini for some fruitful discussions on \cite{Mag}.

\end{document}